\documentclass{article}

\RequirePackage[%lnms,
amsthm,amsmath%,natbib
]{imsart}

\usepackage{graphicx}
\usepackage{latexsym,amsmath}
\usepackage{amsmath,amsthm,amscd}
\usepackage{amsfonts}
\usepackage[psamsfonts]{amssymb}
\usepackage{enumerate}

\usepackage{url}

\usepackage{tocvsec2}

\usepackage{epstopdf}
\usepackage{wrapfig}
\usepackage{float}

\startlocaldefs
\theoremstyle{plain} 
\newtheorem{theorem}{Theorem}[section]
\newtheorem{corollary}[theorem]{Corollary}

\newtheorem{lemma}[theorem]
{Lemma}

\newtheorem{proposition}[theorem]{Proposition}

\theoremstyle{definition} 

\theoremstyle{definition} 

\newtheorem*{ex*}{Example}
\theoremstyle{remark} 

\theoremstyle{remark} 

\newtheorem*{remark*}{Remark}
\newcommand{\beqa}{\begin{eqnarray}}
\newcommand{\eeqa}{\end{eqnarray}}
 
\newcommand{\bseq}{\begin{subequations}}
 
\newcommand{\eseq}{\end{subequations}}

\newcommand{\dd}{\partial}

\renewcommand{\dd}{{\,\operatorname{d}}}

\newcommand{\Ga}{\Gamma}

\newcommand{\la}{\lambda}

\renewcommand{\th}{\theta}
\renewcommand{\Psi}{\overline{\Phi}}

\newcommand{\ii}[1]{\,\mathbf{I}\{#1\}} 

\newcommand{\pd}[2]{\frac{\partial#1}{\partial#2}} 
\newcommand{\fd}[2]{\frac{\dd#1}{\dd#2}}

\newcommand{\E}{\operatorname{\mathsf{E}}}

\newcommand{\R}{\mathbb{R}}

\renewcommand{\d}{\mathrm{d}}

\renewcommand{\le}{\leqslant}

\endlocaldefs

\begin{document}

\begin{frontmatter}

\title{Monotone tail and moment ratio properties of Student's family of distributions}
%\title{Tail monotonicity properties of Student's family of distributions}
\runtitle{Monotonicity properties of Student's family}
%\date{\today}

% \author{\fnms{First}  \snm{Author}\corref{}\thanksref{t2}\ead[label=e1]{first@somewhere.com}},
%  \author{\fnms{Second} \snm{Author}\ead[label=e2]{second@somewhere.com}}
%  \and
%  \author{\fnms{Third}  \snm{Author}%
%  \ead[label=e3]{third@somewhere.com}%
%  \ead[label=u1,url]{http://www.foo.com}}
%
%  \thankstext{t2}{Footnote to the first author with the `thankstext' command.}

\begin{aug}
\author{\fnms{Iosif} \snm{Pinelis}\thanksref{t2}\ead[label=e1]{ipinelis@mtu.edu}}
  \thankstext{t2}{Supported in part by NSF grant DMS-0805946 and NSA grant H98230-12-1-0237}
\runauthor{Iosif Pinelis}

%\affiliation{Michigan Technological University}

\address{Department of Mathematical Sciences\\
Michigan Technological University\\
Houghton, Michigan 49931, USA\\
E-mail: \printead[ipinelis@mtu.edu]{e1}}
\end{aug}

\begin{abstract}
Let $G_p$ denote the tail function of Student's distribution with $p$ degrees of freedom. 
%, for $p\in(0,\infty]$. 
It is shown that the ratio $G_q(x)/G_p(x)$ is decreasing in $x>0$ for any $p$ and $q$ such that $0<p<q\le\infty$. 
Therefore, $G_q(x)<G_p(x)$ for all such $p$ and $q$ and all $x>0$. 
Corollaries on the monotonicity of (generalized) moments and ratios thereof are also given. 
\end{abstract}

%\subjclass[2000]{60E15, 62G10, 62G15, 60G50, 62G35}
% 62G10    	Hypothesis testing
%  62G15    	Tolerance and confidence regions
%  60G50    	Sums of independent random variables; random walks
%   62G35    	Robustness
  
%
%\keywords{probability inequalities; Rade\-macher random variables; sums of independent random variables; Student's test; self-normalized sums}

\begin{keyword}[class=AMS]
\kwd[Primary ]{60E15}
%\kwd{60B11}
%\kwd{62G10}
\kwd[; secondary ]{62E15}
%\kwd{}
\kwd{62F03}
\end{keyword}

%60E15   	Inequalities; stochastic orderings
%62Exx 		Distribution theory [See also 60Exx]
%62E10   	Characterization and structure theory
%	62E15   	Exact distribution theory
%	62G10   	Hypothesis testing
%	62Fxx 		Parametric inference
%		62F03   	Hypothesis testing

\begin{keyword}
\kwd{Student's distribution}
\kwd{stochastic monotonicity}
\kwd{monotone tail ratio}
\kwd{probability inequalities}
\end{keyword}

\end{frontmatter}

\settocdepth{chapter}

%\tableofcontents 
%%%%%%%%%%%%%%%%%{\small\tableofcontents} 

\settocdepth{subsubsection}

\theoremstyle{plain} 
%\newtheorem{theorem}{Theorem}[section]
%\newtheorem{corollary}[theorem]{Corollary}
%\newtheorem*{main}{Main~Theorem}
%\newtheorem{lemma}{Lemma}[subsection]
%\newtheorem{proposition}[theorem]{Proposition}
%\newtheorem{conjecture}{Conjecture}
%\theoremstyle{definition} 
%\newtheorem{definition}[theorem]{Definition}
%\theoremstyle{definition} 
%\newtheorem{ex}{Example}
%\theoremstyle{remark} 
%\newtheorem{exer}{Exercise}
%\theoremstyle{remark} 
%\newtheorem{remark}[theorem]{Remark}b
%\newtheorem*{remark*}{Remark}

%\numberwithin{equation}{section}

%
%\eject
%
%
\section{Summary and discussion}\label{intro} 
%
%\subsection{Summary}\label{summary} 
The density and tail functions of Student's distribution with $p$ degrees of freedom are given, respectively, by the formulas 
\begin{align}
	f_p(x)&:=\frac{ \Gamma
   \left(\frac{p+1}{2}\right)}{\sqrt{\pi p}\, \Gamma \left(\frac{p}{2}\right)}\,
   \left(1+\frac{x^2}{p}\right)^{-(p+1)/2} \quad\text{and} \label{eq:f_p}\\
   G_p(x)&:=\int_x^\infty f_p(u)\dd u. \notag 
\end{align}
for all real $x$. 
Most often, the values of the parameter $p$ are assumed to be positive integers. 
However, formula \eqref{eq:f_p} defines a probability density function for all real $p>0$, and, as we shall see, it may be advantageous, at least as far as proofs are concerned, to let $p$ take on all positive real values.  
Let us also extend these definitions by continuity to $p=\infty$, so that $f_\infty$ and $G_\infty$ are the density and tail functions of the standard normal distribution. 

%\newpage

Let $I$ be an interval on the real line. A family $(f_\th)_{\th\in I}$ of (say everywhere strictly positive) probability density functions is said to have a monotone likelihood ratio (MLR) if, for any $\th_0$ and $\th_1$ in $I$ such that $\th_0<\th_1$, the ratio $f_{\th_1}/f_{\th_0}$ is (say strictly) increasing on $\R$. 
Let $(G_\th)_{\th\in I}$ denote the family of the corresponding tail functions, so that $G_\th(x):=\int_x^\infty f_\th(u)\dd u$ for all $\th\in I$ and $x\in\R$. 
It is well known that the MLR property implies the stochastic monotonicity (SM): $G_{\th_1}(x)>G_{\th_0}(x)$ for all real $x$ and all $\th_0$ and $\th_1$ in $I$ such that $\th_0<\th_1$. 
In fact, one can say more. 
Introduce also the monotone tail ratio (MTR) property, meaning that, for any $\th_0$ and $\th_1$ in $I$ such that $\th_0<\th_1$, the ratio $G_{\th_1}/G_{\th_0}$ is strictly increasing on $\R$. 
Then one has 
\begin{equation*}
	\text{MLR}\implies\text{MTR}\implies\text{SM};  
\end{equation*}
the ``non-strict'' version of these implications was given in 
\cite{keilson-sumita}. 
Stochastic monotonicity is important in deriving uniformly most powerful tests; see e.g.\ \cite{karlin-rubin-AnnMathStat56}. 
Generally, SM is derived based on MLR, but not not necessarily via MTR.  

It is not hard to see that the Student family of densities does not have the MLR property; see \eqref{eq:fq/fp} below. 
However, we shall show here that this family (strictly speaking, parameterized by $-p$ rather than by $p$) still has the MTR and hence SM properties; 
a key here is a general l'Hospital-type rule for monotonicity; see e.g.\ \cite{monthly,pin06} and further references there. 

\begin{theorem}\label{th:}
For any $p$ and $q$ such that $0<p<q\le\infty$, the tail ratio 
\begin{equation}\label{eq:decr}
\frac{G_q(x)}{G_p(x)}	\ \text{is (strictly) decreasing in $x\in[0,\infty)$}, 
\end{equation}
which implies the (strict) stochastic majorization:  
\begin{equation}\label{eq:less}
	G_q(x)<G_p(x)	\ \text{for all $x>0$}. 
\end{equation}
\end{theorem}

Note that \eqref{eq:decr} can be rewritten as follows: 
\begin{equation*}%\label{eq:decr}
\frac{G_p(v)}{G_p(u)}	\ \text{is decreasing in $p\in(0,\infty]$ whenever $0\le u<v<\infty$;}  
\end{equation*}
this can also be expressed, for instance, as the statement that the function $(0,\infty]\times(-\infty,0]\ni(p,y)\mapsto G_p(-y)$ is strictly totally positive of order $2$ (STP${}_2$); see e.g.\ \cite{karlin_total-pos}. 

Let us now present some corollaries of Theorem~\ref{th:} concerning (possibly generalized) moments and ratios of moments. 

For any $p\in(0,\infty]$, let $T_p$ denote any random variable (r.v.) which has Student's distribution with $p$ degrees of freedom. In particular, $T_\infty$ will have the standard normal distribution. 

Let $b\colon[0,\infty)\to[0,\infty)$ be a nondecreasing function, which is non-constant on $(0,\infty)$. %, which is nonzero -- that is, does not . 

%\newpage

\begin{corollary}\label{cor:1} \ %Let $a\colon[0,\infty)\to[0,\infty)$ be any nondecreasing function. Then one has the following. 
\begin{enumerate}[(i)]
	\item The (generalized) $b$-moment $\E b(|T_p|)$ of the r.v.\ $|T_p|$ is 
	nonincreasing in $p\in(0,\infty]$, and so, 
	the ``finiteness'' set 	
\begin{equation}
F_b:=\{p\in(0,\infty]\colon\E b(|T_p|)<\infty\}	
\end{equation}
of the $b$-moment 
is an interval of the form $[p_b,\infty)$ or $(p_b,\infty)$, for some $p_b\in(0,\infty]$. 
Moreover, $\E b(|T_p|)$ is (strictly) decreasing in $p\in[p_b,\infty)$. 
	\item In particular,  for any given $s\in(0,\infty)$, the moment $\E|T_p|^s$ is decreasing in $p\in[s,\infty]$ (of course, $\E|T_s|^s=\infty$). 
\end{enumerate}
\end{corollary}

Note that the ``finiteness'' set $F_b$ can actually be of either form: $[p_b,\infty)$ or $(p_b,\infty)$. E.g., if $b(x)=x^s$ for any $s\in(0,\infty)$ and all $x\in[0,\infty)$ then $F_b=(s,\infty)$; and if $b(x)=x^s/\ln^2(e^{2/s}+x)$ for any $s\in(0,\infty)$ and all $x\in[0,\infty)$ then $F_b=[s,\infty)$. 

Let now $a\colon[0,\infty)\to[0,\infty)$ and $b\colon[0,\infty)\to[0,\infty)$ be nonzero nondecreasing right-continuous functions; we say that a function is nonzero on a given set if it is not identically zero on that set. Let $\mu_a$ and $\mu_b$ be the corresponding Lebesgue--Stieltjes measures on $[0,\infty)$, defined by the conditions $\mu_a([0,x])=a(x)$ and $\mu_b([0,x])=b(x)$ for all $x\in[0,\infty)$. 
Suppose that the measure $\mu_b$ is absolutely continuous with respect to $\mu_a$, with a density 
\begin{equation*}
	\rho:=\rho_{a,b}:=\frac{\d b}{\d a}. 
\end{equation*}
In particular, if the functions $a$ and $b$ are continuous on $[0,\infty)$ and continuously differentiable on $(0,\infty)$ with %$a(0)=b(0)=0$ and 
the derivative $a'>0$ on $(0,\infty)$, then one can take $\rho=b'/a'$ on $(0,\infty)$; if at that $a(0)\ne0$ then $\rho(0)=b(0)/a(0)$. 

Given a (nonnegative Borel) measure $\mu$ on $[0,\infty)$, a Borel set $E\subseteq[0,\infty)$, 
and a real-valued Borel function $h$, defined on a subset of $[0,\infty)$ containing $E$,
let us say that $h$ is $\mu$-constant on $E$ if $\mu\big(\{x\in E\colon h(x)\ne \ell\}\big)=0$ for some $\ell\in\R$. 
If $h$ is not $\mu$-constant on the domain of definition of $h$, let us just say that $h$ is not $\mu$-constant. 

%\newpage

\begin{corollary}\label{cor:2}\ 
\begin{enumerate}[(i)]
	\item If $\rho$ is nondecreasing on $[0,\infty)$ then the ratio $\dfrac{\E b(|T_p|)}{\E a(|T_p|)}$ is nonincreasing in $p\in F_a\cap F_b$; moreover, this ratio is decreasing in $p\in F_a\cap F_b$ if $\rho$ is not $\mu_a$-constant on $[0,\infty)$.  
	\item In particular, the ratio $\dfrac{\E|T_p|^t}{\E|T_p|^s}$ is decreasing in $p\in(t,\infty]$ for any given real numbers $s$ and $t$ such that $0<s<t$. 
\end{enumerate}
\end{corollary}

Note that the generalized moment $\E a(|T_p|)$ in the denominator of the ratio in part (i) of Corollary~\ref{cor:2} is (strictly) positive, since the function $a$ was assumed to be nonzero, nonnegative, and nondecreasing on $[0,\infty)$.

As usual, let $x_+:=0\vee x$% and $x_-:=(-x)_+$ for all real $x$
, and let $\la$ stand for the Lebesgue measure. 

%\newpage

One also has the following proposition, based on the MLR property stated in \eqref{eq:fq/fp}. The conventions on the functions $a$ and $b$ made before Corollaries~\ref{cor:1} and \ref{cor:2} will not necessarily apply in what follows. 

\begin{proposition}\label{prop:1}\ 
\begin{enumerate}[(i)]
	\item 
Let $b\colon[0,1]\to[0,\infty)$ be a nondecreasing function, which is not $\la$-constant. % on $[0,1]$. 
Then the conditional expectation 
$\E\big(b(1-|T_p|)\,\big|\,|T_p|<1\big)$ is %nonin
decreasing in $p\in(0,\infty]$.  
	\item 
Let $b\colon[0,\infty)\to[0,\infty)$ be a bounded nondecreasing function, which is not $\la$-constant. 
% on $[0,\infty)$. 
Then %\break 
$\E\big(b(|T_p|-1)\,\big|\,|T_p|>1\big)$ is 
decreasing in $p\in(0,\infty]$. 
	\item 
Let $a\colon[0,1]\to[0,\infty)$ be a bounded Borel function with $a(0)=0$.  
Let $r\colon[0,1]\to[0,\infty)$ be a nondecreasing function, which is not $\la$-constant on the set $\{t\in[0,1]\colon a(t)\ne0\}$.   
Let then $b:=ra$.  
Then the ratio $\dfrac{\E b\big((1-|T_p|)_+\big)}{\E a\big((1-|T_p|)_+\big)}$ is decreasing in $p\in(0,\infty]$.   
	\item 
Let $a\colon[0,\infty)\to[0,\infty)$ be a bounded Borel function with $a(0)=0$.    
Let $r\colon[0,\infty)\to[0,\infty)$ be a bounded nondecreasing function, which is not $\la$-constant on the set %\break 
$\{t\in[0,\infty)\colon a(t)\ne0\}$.   
Let then $b:=ra$. 
Then the ratio $\dfrac{\E b\big((|T_p|-1)_+\big)}{\E a\big((|T_p|-1)_+\big)}$ 
is decreasing in $p\in(0,\infty]$. 
\end{enumerate}
\end{proposition}

Let us also present a related result for the $\chi^2$ distribution, based on the MLR property given in \cite[Theorem~2.9]{T2}. 
For $p\in[1,\infty)$, let $Z_p:=X_p-\sqrt{p-1}$, where $X_p$ is any nonnegative r.v.\ such that $X_p^2$ has the $\chi^2$ distribution with $p$ degrees of freedom; complete this definition by continuity, letting $Z_\infty$ have the centered normal distribution with variance $\frac12$. 

\begin{proposition}\label{prop:2}
\ 
\begin{enumerate}[(i)]
	\item 
Let $b\colon\R\to[0,\infty)$ be a bounded nondecreasing function, which is non-constant on $[0,\infty)$. Then $\E b(Z_p)$ is %nonin
decreasing in $p\in[1,\infty]$. 
	\item 
Let $a\colon\R\to[0,\infty)$ be a bounded Borel function.   
Let $r\colon\R\to[0,\infty)$ be a bounded nondecreasing function, which is not $\la$-constant on the set $\{t\in[0,\infty)\colon %\break
a(t)\ne0\}$. 
Let then $b:=ra$. 
Then the ratio $\dfrac{\E b(Z_p)}{\E a(Z_p)}$ is decreasing in $p\in[1,\infty]$.    
\end{enumerate}
\end{proposition} 

Note that the generalized moment $\E a\big((1-|T_p|)_+\big)$ in the denominator of the decreasing ratio in part (iii) of Proposition~\ref{prop:1} %and in part (ii) of Proposition~\ref{prop:2} 
is (strictly) positive, because the function $a$ is nonnegative on $[0,1]$ and the function $r$ is not $\la$-constant on the subset of $[0,1]$ where $a\ne0$, which implies that the latter subset is of nonzero Lebesgue measure; on the other hand, the probability that the r.v.\ $(1-|T_p|)_+$ takes a value in any such  subset is nonzero. 
For similar reasons, the generalized moments in the denominators of the decreasing ratios in part (iv) of Proposition~\ref{prop:1} and in part (ii) of Proposition~\ref{prop:2} 
are (strictly) positive as well. 

As for the boundedness conditions in Propositions~\ref{prop:1} and \ref{prop:2}, they are used as a simplest way to ensure that the $a$- and $b$-moments in these propositions are finite. 
Note the following: 

\begin{enumerate}
	\item In part (i) of Proposition~\ref{prop:1}, the function $b$ will be automatically bounded, since it is nondecreasing and defined on a compact interval; similarly for the function $r$ in part (iii) of Proposition~\ref{prop:1}. This is in contrast with the other parts of Proposition~\ref{prop:1}, as well as with Proposition~\ref{prop:2}. 
	\item Part (ii) of Proposition~\ref{prop:1} and part (i) of Proposition~\ref{prop:1}, where the function $b$ is assumed to be nondecreasing, can be replaced by somewhat more complicated statements similar to part (i) of Corollary~\ref{cor:1}. 
	\item The boundedness condition in the other parts of Propositions~\ref{prop:1} and \ref{prop:2} can be relaxed as well, but then the statements and proofs will be more complicated, and in some cases the strictness of the decrease may be lost. 
\end{enumerate}

Theorem~\ref{th:} was motivated by the study in \cite{closeness-student} of closeness of the members of the Student family of distributions to one another and, in particular, to the standard normal distribution. 
Indeed, the results of the present note --- Theorem~\ref{th:} and Lemma~\ref{lem:} --- are used in the proof of the main results in \cite{closeness-student}. 
Corollaries~\ref{cor:1}  and \ref{cor:2} were motivated by work on the paper \cite{nonlinear} and are used there.  

\section{Proofs}\label{proofs} 

The proof of Theorem~\ref{th:} is based in part on 

\begin{lemma}\label{lem:}
For all $p\in(0,\infty)$ one has $\fd{}p f_p(0)>0$. 
\end{lemma}

\begin{proof}
Take indeed any $p\in(0,\infty)$. 
Using the Gauss integral formula for $\psi:=\Ga'/\Ga$ (see e.g.\ \cite[Theorem~1.6.1]{andrews}), then changing the integration variable to $t:=e^{-z/2}$ and finally integrating by parts, one has 
\begin{align*}
	2p\,\fd{\ln f_p(0)}p
	&=-p \psi\Big(\frac{p}{2}\Big)+p \psi\Big(\frac{p+1}{2}\Big)-1 \\
	&=p\int_0^\infty \frac{e^{-pz/2}-e^{-(p+1)z/2}}{1-e^{-z}}\dd z -1 \\ 
	&=\int_0^1\frac{2pt^{p-1}\dd t}{1+t} - 1
	=\int_0^1\frac{2t^p\dd t}{(1+t)^2}>0. 
\end{align*}
So, $\fd{}p\ln f_p(0)>0$ and hence $\fd{}p f_p(0)>0$, which proves the lemma. 
\end{proof}

%Theorem~\ref{th:} was motivated by the study in \cite{closeness-student} of closeness of the members of the Student family of distributions to one another and, in particular, to the standard normal distribution. 
%Indeed, the results of the present note --- Theorem~\ref{th:} and Lemma~\ref{lem:} --- are used in the proof of the main results in \cite{closeness-student}. 

\begin{proof}[Proof of Theorem~\ref{th:}]
Take any positive $p$ and introduce
\begin{gather*}
	r(x):=r_p(x):=\pd{\ln G_p(x)}p=\frac{f(x)}{g(x)},\quad\text{where}\\
	f(x):=\pd{G_p(x)}p\quad\text{and}\quad g(x):=G_p(x),  
\end{gather*}
for all real $x$. 
Then 
\begin{equation*}
	\rho(x):=\frac{f'(x)}{g'(x)}=\frac{\frac{\partial^2}{\partial p\,\partial x}G_p(x)}{\pd{}x G_p(x)}
	=\frac{\pd{}p f_p(x)}{f_p(x)}
	=\pd{\ln f_p(x)}p. 
\end{equation*}
So, 
\begin{equation}\label{eq:rho'}
		\rho'(x)
	=\frac{\partial^2 \ln f_p(x)}{\partial p\,\partial x}
		=\frac{x(1-x^2)}{(p+x^2)^2}. 
\end{equation}
It follows that $\rho$ is increasing on $(0,1]$ and decreasing on $[1,\infty)$. 
Note that $r'(0)=-2\fd{f_p(0)}p$. 
So, by Lemma~\ref{lem:}, $r'(0)<0$. 
Using now \cite[Proposition 4.3]{pin06}, one sees that $r$ is decreasing on $(0,\infty)$; that is, 
$\pd{\ln G_p(x)}p$ is decreasing in $x>0$. 
Therefore, for any $p$ and $q$ such that $0<p<q\le\infty$ 
\begin{equation}\label{eq:Gq/Gp}
	\ln\frac{G_q(x)}{G_p(x)}=\ln G_q(x) - \ln G_p(x)=\int_p^q \pd{\ln G_s(x)}s \dd s 
\end{equation}
is decreasing in $x>0$; so, the statement \eqref{eq:decr} holds, which in turn implies 
$\frac{G_q(x)}{G_p(x)}<\frac{G_q(0)}{G_p(0)}=1$ for $x>0$ and hence \eqref{eq:less}. 
\end{proof}

One may also note that \eqref{eq:rho'} implies (cf.\ \eqref{eq:Gq/Gp}) that 
\begin{equation}\label{eq:fq/fp}
	\text{$\frac{f_q(x)}{f_p(x)}$ is increasing in $x\in[0,1]$ and decreasing in $x\in[1,\infty)$} 
\end{equation}
--- again for any $p$ and $q$ such that $0<p<q\le\infty$.

\begin{proof}[Proof of Corollary~\ref{cor:1}] \  
%\noindent\textbf{(i)}\quad 
Using integration by parts or, more precisely, the Fubini theorem, one has 
\begin{align}
	\frac12\,\E b(|T_p|)&=\int_0^\infty f_p(x)b(x)\d x %\label{eq:fubini1} %\\
	%&
	=\int_{[0,\infty)}G_p(x)\mu_b(\d x)  \label{eq:fubini}
\end{align}
for all $p\in(0,\infty]$, 
where, as before, $\mu_b$ is the Lebesgue--Stieltjes measure corresponding to the function $b$. 
So, by \eqref{eq:less}, $\E b(|T_p|)$ is indeed nondecreasing in $p\in(0,\infty]$, which indeed implies that 
$F_b$ is of the form $[p_b,\infty)$ or $(p_b,\infty)$. 
Moreover, since $b$ was assumed to be non-constant on $(0,\infty)$, one has $\mu_b\big((0,\infty)\big)>0$, which, again by \eqref{eq:less}, implies that $\E b(|T_p|)$ is decreasing in $p\in F_b$. In the case when $F_b=[p_b,\infty)$, this completes the proof of part (i) of Corollary~\ref{cor:1}. 
The same conclusion obtains in the case when $F_b=(p_b,\infty)$, because then $\E b(|T_{p_b}|)=\infty>\E b(|T_p|)$ for all $p\in(p_b,\infty)$. 
%
% 
%Moreover, $\E b(|T_p|)$ is (strictly) decreasing in $p\in[p_b,\infty)$. 
%
%Since $b$ was assumed to be non-constant on $(0,\infty)$, one has $\mu_b\big((0,\infty)\big)>0$. Now part (i) of Corollary~\ref{cor:1} follows by \eqref{eq:fubini} and \eqref{eq:less}, and 

As for part (ii) of Corollary~\ref{cor:1}, it is indeed a special case of part (i). 
\end{proof}

We shall need the following folklorish 

\begin{lemma}\label{lem:mu-const}
If $h$ is not $\mu$-constant on $E$, then there exists some $\ell\in\R$ such that $\mu\big(\{x\in E\colon h(x)<\ell\}\big)>0$ and $\mu\big(\{x\in E\colon h(x)>\ell\}\big)>0$. 
\end{lemma}

\begin{proof}%[Proof of Lemma~\ref{lem:mu-const}]
For any $\ell\in\R$, let $H(\ell):=\mu\big(\{x\in E\colon h(x)\le \ell\}\big)$, which is nondecreasing in $\ell$, with $H(-\infty+)=0$ and $H(\infty-)=\mu(E)$. Note also that $\mu(E)>0$, since $h$ is not $\mu$-constant on $E$. 
So, each of the sets $A_1:=\{\ell\in\R\colon H(\ell)>0\}$ and $A_2:=\{\ell\in\R\colon H(\ell)<\mu(E)\}$ is nonempty. 

To obtain a contradiction, suppose that $A_1\cap A_2=\emptyset$. Then the complementary sets $A_1^c=\{\ell\in\R\colon H(\ell)=0\}$ and $A_2^c=\{\ell\in\R\colon H(\ell)=\mu(E)\}$ form a partition of $\R$ and are each nonempty. Moreover, $\ell_1<\ell_2$ for any $\ell_1\in A_1^c$ and $\ell_2\in A_2^c$, because the function $H$ is nondecreasing. So, there is some $\ell_*\in\R$ such that $(-\infty,\ell_*)\subseteq A_1^c$ and $(\ell_*,\infty)\subseteq A_2^c$.  
Hence, $\mu\big(\{x\in E\colon h(x)<\ell_*\}\big)=H(\ell_*-)=0$ and, similarly, 
$\mu\big(\{x\in E\colon h(x)>\ell_*\}\big)=\mu(E)-H(\ell_*+)=0$; 
thus, $\mu\big(\{x\in E\colon h(x)\ne\ell_*\}\big)=0$, which 
does contradict the condition that $h$ is not $\mu$-constant on $E$. 

So, $A_1\cap A_2\ne\emptyset$; that is, $0<H(k)<\mu(E)$ for some $k\in\R$. Since $H$ is nondecreasing and right-continuous, one has $0<H(k)\le H(\ell-)\le H(\ell)<\mu(E)$ for some $\ell\in(k,\infty)$.  
Therefore, $\mu\big(\{x\in E\colon h(x)<\ell\}\big)=H(\ell-)>0$ and $\mu\big(\{x\in E\colon h(x)>\ell\}\big)=\mu(E)-H(\ell)>0$.  
\end{proof}

\begin{proof}[Proof of Corollary~\ref{cor:2}] \  
The idea of this proof is quite standard and goes back to the Chebyshev inequality that may be stated as the nonnegativity of the correlation between two increasing functions of the same r.v. 
We supply details for the readers' convenience and because of the particular concern about the strict monotonicity in the specific situations considered here. 

Take any $p$ and $q$ in $F_a\cap F_b$ such that $p<q$. 
Using \eqref{eq:fubini} and $\d\mu_b=\rho\,\d\mu_a$, one can check the identity 
\begin{multline*}
	\Big(\dfrac{\E b(|T_q|)}{\E a(|T_q|)}-\dfrac{\E b(|T_p|)}{\E a(|T_p|)}\Big)
	\E a(|T_p|)\E a(|T_q|) \\ 
	=
	2\,\int_0^\infty\int_0^\infty[\rho(v)-\rho(u)]\Big[\frac{G_q(v)}{G_p(v)}-\frac{G_q(u)}{G_p(u)}\Big]	
	G_p(u)G_p(v)\mu_a(\d u)\mu_a(\d v), 
\end{multline*}
say by expanding the product of the expressions in the brackets on the right-hand side. 
Now part (i) of Corollary~\ref{cor:2} follows if one refers to \eqref{eq:decr} and also to Lemma~\ref{lem:mu-const};  
and
part (ii) is a special case of part (i). 
\end{proof}

\begin{proof}[Proof of Proposition~\ref{prop:1}] \ 
The proof of part (iii) of Proposition~\ref{prop:1} is quite similar to that of part (i) of Corollary~\ref{cor:2}. 
Use here the identity $\frac12\,\E b\big((1-|T_p|)_+\big)=\int_0^1f_p(x)\,b(1-x)\d x$ (instead of \eqref{eq:fubini}), and similarly with $a$ in place of $b$. 
Thus, use $f_p(x)$, $a(1-x)\d x$, $b(1-x)\d x=r(1-x)a(1-x)\d x$, and the ``increasing'' part of \eqref{eq:fq/fp}  
instead of 
%the expression for $\frac12\,\E b(|T_p|)$ in 
%\eqref{eq:fubini}, 
$G_p(x)$, $\mu_a(\d x)$, $\mu_b(\d x)=\rho(x)\mu_a(\d x)$ 
and \eqref{eq:decr}, respectively.  

The proof of part (iv) of Proposition~\ref{prop:1} is quite similar to the one given just above for part (iii); of course, here one will use the ``decreasing'' part of \eqref{eq:fq/fp}. 

As for part (i) of Proposition~\ref{prop:1}, there without loss of generality $b(0)=0$. Then part (i) becomes a special case of part (iii), with $a(x)=\ii{x\in(0,1]}$ and $r(x)=b(x)$ for all $x\in[0,1]$. 
Similarly, part (ii) of Proposition~\ref{prop:1} reduces to a special case of part (iv). 
\end{proof} 

\begin{proof}[Proof of Proposition~\ref{prop:2}] \  
The proofs of parts (i) and (ii) of Proposition~\ref{prop:2} are quite similar to those of part (i) of Corollary~\ref{cor:1} and part (i) of Corollary~\ref{cor:2}, respectively. 
Here one should also slightly modify the statement and proof of the mentioned Theorem~2 in \cite{T2} in order to prove the strict decrease. 
\end{proof}

\bibliographystyle{abbrv}
%\bibliographystyle{ims}
%\bibliography{are.citations}
%\bibliography{citat}

%\bibliography{citations}

\bibliography{C:/Users/Iosif/Dropbox/mtu/bib_files/citations}
%\bibliography{C:/Users/Iosif/Documents/mtu_home01-30-10/bib_files/citations}
%\bibliography{C:/Users/Iosif/Documents/mtu_home12-22-08/bib_files/citations}

\end{document}